\documentclass[12pt]{article}
\usepackage{enumerate}
\usepackage{amssymb,amsmath,amsfonts,amsthm,mathrsfs}
\usepackage[colorlinks=true, pdfstartview=FitV, linkcolor=blue, citecolor=blue, urlcolor=blue]{hyperref}
\usepackage{graphicx,color}
\usepackage{cite}
\usepackage{indentfirst}

\allowdisplaybreaks

\begin{document}
\def\rn{{\mathbb R^n}}  \def\sn{{\mathbb S^{n-1}}}
\def\co{{\mathcal C_\Omega}}
\def\z{{\mathbb Z}}
\def\nm{{\mathbb (\rn)^m}}
\def\mm{{\mathbb (\rn)^{m+1}}}
\def\n{{\mathbb N}}
\def\cc{{\mathbb C}}

\newtheorem{defn}{Definition}
\newtheorem{thm}{Theorem}
\newtheorem{lem}{Lemma}
\newtheorem{cor}{Corollary}
\newtheorem{rem}{Remark}

\title{\bf\Large Mixed radial-angular bounds for a class of integral operators
on Heisenberg groups
\footnotetext{
\newline\indent\hspace{1mm} {\it 2020 Mathematics Subject Classification}: Primary 42B25; Secondary 42B20, 47H60, 47B47.}}

\date{}
\author{Xiang Li, Huan Liang, Shaozhuang Xu\footnote{Corresponding author}, Dunyan Yan}
\maketitle
\begin{center}
\begin{minipage}{13cm}
{\small {\bf Abstract:}\quad
In this paper, we will prove the sharp bounds of various operators in mixed radial angular spaces on Heisenberg groups. It mainly includes the boundedness of linear transformation eigenvalue operator in mixed radial angular space; Sharp Bounds of Hilbert Operator and Hardy-Littlewood-P$\acute{o}$lya Operator in mixed radial-angular space.}
\end{minipage}
\end{center}

%------------------------section 1------------------------
\section{Introduction}
\par
\begin{thm}\label{thm1}
Suppose $K$ is a non-negative kernel defined on $\mathbb{R}^n \times \mathbb{R}^n$, continuous on any domain that excludes the point $(0,0)$, homogeneous of degree $-n, K(\delta u, \delta v)=\delta^{-n} K(u, v)$and $K(R u, R v)=$ $K(u, v)$ for any $R \in S O(n)$. Then $K$ define an integral operator
$$
T f(x)=\int_{\mathbb{R}^n} K(x, y) f(y) d y
$$
which maps $L^p\left(\mathbb{R}^n\right)$ to $L^p\left(\mathbb{R}^n\right)$ for $1<p<\infty$
\begin{equation}\label{main_1}
\|T f\|_{L^p\left(\mathbb{R}^n\right)} \leq C\|f\|_{L^p\left(\mathbb{R}^n\right)}
\end{equation}
where the optimal constant is given by
\begin{equation}\label{main_2}
C=\int_{\mathbb{R}^n} K\left(x, \hat{e}_1\right)|x|^{-n / p^{\prime}} d x
\end{equation}
and $\hat{e}_1$ is a unit vector in the first coordinate direction.
 \end{thm}
The inequality (\ref{main_1}) was first given by Stein and Weiss \cite{EM}. In \cite{BW1,BW2}, Beckner et al. proved that the constant (\ref{main_2}) is the sharp constant. Inspired by Theorem \ref{thm1}, we investigate a class of integral operators on Heisenberg group and obtain sharp bounds in Mixed radial-angular space.
Firstly,let us recall some basic knowledge about Heisenberg group.The Heisenberg group $\mathbb{H}^n$ is non-commutative nilpotent Lie group, with the underlying manifold $\mathbb{R}^{2n+1}$ and the group law.

Let
$$
x=(x_1,\ldots,x_{2n},x_{2n+1}),y=(y_1,\ldots,y_{2n},y_{2n+1}),
$$
then we have
$$
x \circ y=\left(x_1+y_1, \ldots, x_{2 n}+y_{2 n}, x_{2 n+1}+y_{2 n+1}+2 \sum_{j=1}^n(y_j x_{n+j}-x_j y_{n+j})\right).
$$
By definition, we can see that the identity element on $\mathbb{H}^n$ is $0 \in \mathbb{R}^{2 n+1}$, while the element $x^{-1}$ inverse to $x$ is $-x$. The corresponding Lie algebra is generated by the left-invariant vector fields
$$
\begin{gathered}
	X_j=\frac{\partial}{\partial x_j}+2 x_{n+j} \frac{\partial}{\partial x_{2 n+1}}, \quad j=1,2, \ldots, n, \\
	X_{n+j}=\frac{\partial}{\partial x_{n+j}}-2 x_j \frac{\partial}{\partial x_{2 n+1}}, \quad j=1,2, \ldots, n, \\
	X_{2 n+1}=\frac{\partial}{\partial x_{2 n+1}} .
\end{gathered}
$$
The only non-trivial commutator relation is
$$
\left[X_j, X_{n+j}\right]=-4 X_{2 n+1}, \quad j=1,2 \ldots, n.
$$
Note that  Heisenberg group $\mathbb{H}^n$ is a homogeneous group with dilations
$$
\delta_r(x_1, x_2, \ldots, x_{2 n}, x_{2 n+1})=(r x_1, r x_2, \ldots, r x_{2 n}, r^2 x_{2 n+1}), \quad r>0.
$$

The Haar measure on $\mathbb{H}^n$ coincides with the usual Lebesgue measure on $\mathbb{R}^{2n+1}$. Denoting any measurable set $E \subset \mathbb{H}^n$ by $|E|$, then we obtain
$$
|\delta_r(E)|=r^Q|E|, d(\delta_r x)=r_Q d x,
$$
where $Q=2n+2$ is called the homogeneous dimension of $\mathbb{H}^n$.

The Heisenberg distance derived from the norm
$$
|x|_h=\left[\left(\sum_{i=1}^{2 n} x_i^2\right)^2+x_{2 n+1}^2\right]^{1 / 4},
$$
where $x=(x_1,x_2,\ldots,x_{2n},x_{2n+1})$ is given by
$$
d(p, q)=d(q^{-1} p, 0)=|q^{-1} p|_h.
$$
This distance $d$ is left-invariant in the sense that $d(p,q)$ remains unchanged when $p$ and $q$ are both left-translated by some fixed vector on $\mathbb{H}^n$. Besides, $d$ satisfies the triangular inequality defined by \cite{Kor}
$$
d(p, q) \leq d(p, x)+d(x, q), \quad p, x, q \in \mathbb{H}^n.
$$
For $r>0$ and $x\in\mathbb{H}^n$, the ball and sphere with center $x$ and radius $r$ on $\mathbb{H}^n$ are given by
$$
B(x, r)=\{y \in \mathbb{H}^n: d(x, y)<r\}
$$
and
$$
S(x, r)=\{y \in \mathbb{H}^n: d(x, y)=r\}.
$$
Then we have
$$
|B(x, r)|=|B(0, r)|=\Omega_Q r^Q,
$$
where
$$
\Omega_Q=\frac{2 \pi^{n+\frac{1}{2}} \Gamma(n / 2)}{(n+1) \Gamma(n) \Gamma((n+1) / 2)}
$$
denote the volume of the unit ball $B(0,1)$ on $\mathbb{H}^n$  that is $\omega_Q=Q \Omega_Q$ (see \cite{}). For more details about Heisenberg group can be refer to \cite{GB} and \cite{ST}.

In recent years, mixed radial-angular spaces have been successfully used to study Streehartz estimation and partial differential equations to improve the corresponding results (\cite{CMZ},\cite{AC},\cite{AD},\cite{AE},\cite{AF},\cite{AG}). After that, many operators in harmonic analysis are proved to be bounded on these spaces. For example, \cite{AH} established the outer theorem of mixed radial angular space and studied the boundedness of a class of weighted bounded operators is discussed. In addition, Liu et al. \cite{AI},\cite{AJ},\cite{AK} also considered the boundedness of some operators with rough kernels in mixed radial-angular spaces.
In \cite{FGLY},\cite{FGPW},\cite{FHLL},\cite{FLPS},\cite{FLS},\cite{FPW},\cite{SFL} and \cite{WF1}, Fu and Wu et al. have engaged in many related research, which provid convenience for our research. Hang et al. have also conducted many related studies(see \cite{HXY}).

Now, we give the definition of mixed radial-angular spaces on Heisenberg group.
\begin{defn}
	For any $n\geq2$,$1\leq p$,$\bar{p}\leq\infty$, the mixed radial-angular space $L^p_{|x|_h} L^{\bar{p}}_\theta(\mathbb{H}^n)$ consists of all functions $f$ in $\mathbb{H}^n$ for which
$$	\|f\|_{L^p_{|x|_h}L^{\bar{p}}_\theta(\mathbb{H}^n)}:=\left(\int_0^\infty \left(\int_{\mathbb{S}^{Q-1}}|f(r,\theta)|^pd \theta\right)^{\frac{p}{\bar{p}}}r^{Q-1}dr\right)^{\frac{1}{p}}<\infty,
$$
where $\mathbb S^{Q-1}$ denotes the unit sphere in $\mathbb{H}^n$.
\end{defn}

Next, we begin to study the operator defined in Theorem \ref{thm1} in mixed radial-angular space on Heisenberg group.
\section{Sharp constant for integral operator}
 First, we give the definition of integral operator with the nonnegative kernel $K$ on Heisenberg group.
 \begin{defn}
Suppose that $K$ is a nonnegative kernel defined on $\mathbb{H}^{n}$ and satisfies the homogeneous degree $-n$,
$$K(|\delta|_h x,|\delta|_h y)=|\delta|^{-n}_h K(x,y),
$$
and$$
K(Rx,Ry)=K(x,y)
$$
for any $R\in SO(n)$, then $K$ denotes the kernel of
\begin{equation}\label{main_3}
Hf(x)=\int_{\mathbb{H}^{n}}K(x,y)f(y)d y.
\end{equation}
\end{defn}
\begin{thm}\label{thm2}
For operator $T$, which maps $L_{|x|_h}^pL_\theta^{\bar{p}_2}(\mathbb H^n)$ to $L_{|x|_h}^pL_\theta^{\bar{p}_2}(\mathbb H^n)$.That is
\begin{align*}
\|T(f)(x)\|_{L_{|x|_h}^pL_\theta^{\bar{p}_2}(\mathbb H^n)}\leq E_m.
\end{align*}
The optimal constant is
\begin{align*}
E_m=\omega_Q^{\frac{1}{\bar{p}_2}-\frac{1}{\bar{p}_1}}\int_{\mathbb H^n}K(e,y)|y|_h^{-Q/p}dy\|f_j\|_{L_{|x|_h}^pL_\theta^{\bar{p}_2}(\mathbb H^n)}.
\end{align*}
Moreover, if operator $T$ is bounded ,we can obtain the norm of the operator $T$.\\
That is
\begin{align*}
\|T(f)(x)\|_{L_{|x|_h}^pL_\theta^{\bar{p}_2}(\mathbb H^n)}=\omega_Q^{\frac{1}{\bar{p}_2}-\frac{1}{\bar{p}_1}}\int_{\mathbb H^n}K(e,y)|y|_h^{-Q/p}dy\|f_j\|_{L_{|x|_h}^pL_\theta^{\bar{p}_2}(\mathbb H^n)}.
\end{align*}
\end{thm}
\begin{proof}
\emph{Set $$g'(x)=\frac{1}{\omega_Q}\int_{\mathbb S^{Q-1}}f(\delta_{|x|_h}\theta)d\theta,x\in \mathbb H^n$$,where g is a radial function.}

\begin{align*}
||g||_{L_{\theta}^{\bar{p_{1}}}(\mathbb H^n)}&=\left(\int_0^\infty\left(\int_{\mathbb S^{Q-1}}|g(r,\theta)|^{\bar{p_1}}d\theta\right)^{\frac{p}{\bar{p_1}}}\right)r^{Q-1}dr)^{\frac{1}{p}}\\
&=\left(\int_0^\infty\left(\frac{1}{\omega_Q}|g(r)|^{\bar{p}_1}\right)^{\frac{p}{\bar{p}_1}}r^{Q-1}dr\right)^{\frac{1}{p}},\\
&=\omega_Q^{\frac{1}{\bar{p}_1}}\left({\int_0^{\infty}|g(r)|^{p}r^{Q-1}dr}\right)^{\frac{1}{p}}
\end{align*}
$g(r)$ can be defined as $g(r)=g(x),x\in \mathbb H^n,|x|_h=r$

Using the H\"{o}lder inequality, we can get
\begin{align*}
\|g\|_{L_{|x|_h}^p(\mathbb H^n)}&=\omega_Q^{\frac{1}{\bar{p}_1}}\left(\int_0^\infty \left|\frac{1}{\omega_Q}\int_{\mathbb S^{Q-1}}f(\delta_r\theta)d\theta\right|^{p}r^{Q-1}dr\right)^{\frac{1}{p}}\\
&=\omega_Q^{\frac{1}{\bar{p}_1-1}}\left(\int_0^\infty\left|\int_{\mathbb S^{Q-1}}f(\delta_r\theta)d\theta\right|^pr^{Q-1}dr\right)^{\frac{1}{p}}\\
&\leq\omega_Q^{\frac{1}{p-1}}\left(\int_0^\infty\left(\int_{\mathbb S^{Q-1}}|f(\delta_r\theta)|^{\bar{p}_1}d\theta\right)^{\frac{p}{\bar{p}_1}}\left(\int_{S^{Q-1}}d\theta\right)^{\bar{p}_1}d\theta\right)^{\frac{p}{\bar{p}_1}}r^{Q-1})^{\frac{1}{p}}\\
&=\left(\int_0^\infty\left(\int_{\mathbb S^{Q-1}}|f(\delta_r\theta)^{\bar{p}_1}d\theta\right)^{\frac{p}{\bar{p}_1}}r^{Q-1}\right)^{\frac{1}{p}}\\
&=\|f\|_{L_{|x|_h}^pL_\theta^{\bar{p}_1}(\mathbb H^n)}\\
Tg(x)&=\int_{\mathbb H^{n}}K(x,y)\left(\frac{1}{\omega_Q}\int_{\mathbb S^{Q-1}}f(\delta_{|x|_h}\theta)d\theta\right)dy\\
&=\int_0^\infty\int_{\mathbb S^{Q-1}}K(rx,ry)\left(\frac{1}{\omega_Q}\int_{\mathbb S^Q-1}f(\delta_{|x|_h}\theta)\right)r^{Q-1}dydr\\
&=\int_0^\infty\int_{\mathbb S^{Q-1}}K(rx,ry)f(\delta_r\theta)r^{Q-1}d\theta dr\\
&=Tf(x),
\end{align*}
So we have
\begin{align*}
\frac{\|T(f)\|_{L_{|x|_h}^pL_\theta^{\bar{p}_2}(\mathbb H^n)}}{\|f\|_{L_{|x|_h}^pL_\theta^{\bar{p}_2}(\mathbb H^n)}}\leq\frac{\|T(g)\|_{L_{|x|_h}^pL_\theta^{\bar{p}_2}(\mathbb H^n)}}{\|g\|_{L_{|x|_h}^pL_\theta^{\bar{p}_2}(\mathbb H^n)}},
\end{align*}
Where means that the operator t can be expressed as a radial function, so we can get
\begin{align*}
\|\mathcal H_hf\|_{L_|x|_h^pL_\theta^{\bar{p}_2}(\mathbb H^n)}&=\left(\int_0^\infty\left(\int_{S^{Q-1}}|\mathcal H_h(f)(r,\theta)|^{\bar{p}_2}d\theta\right)^{\frac{p}{\bar{p}_2}}r^{Q-1}dr\right)^{\frac{1}{p}}\\
&=\left(\int_0^\infty\left(\int_{S^{Q-1}}|\mathcal H_h(f)(r)|^{\bar{p}_2}d\theta\right)^{\frac{p}{\bar{p}_2}}r^{Q-1}dr\right)^{\frac{1}{p}}\\
&=\omega_Q^{\frac{1}{\bar{p}_2}}\left(\int_0^\infty|\mathcal H_h(f)(r)|^pdr\right)^{\frac{1}{p}},
\end{align*}
where $T_h(f)(r)=T_H(f)(x),|x|_h=r$.\\
Next, using Minkowski's inequality, we can continue to get
\begin{align*}
\|T(f)\|_{L_{|x|_h}^pL_\theta^{\bar{p}_2}(\mathbb H^n)}&=\omega_Q^{1/{\bar{p}_2}}\left(\int_0^\infty\left|\int_{\mathbb H^n}K(x,y)f(y)dy\right|^pr^{Q-1}dr\right)^{\frac{1}{p}}\\
&=\omega_Q^{1/{\bar{p}_2}}\left(\int_0^\infty\left|\int_{\mathbb H^n}K(e,y)f(\delta_ry)dy\right|^pr^{Q-1}dr\right)^{\frac{1}{p}}\\
&\leq\omega_Q^{1/{\bar{p}_2}}\int_{\mathbb H^n}K(e,y)\left(\int_0^\infty|f(\delta_{|y|}r)|pr^{Q-1}dr\right)^{\frac{1}{p}}dy\\
&=\omega_Q^{1/{\bar{p}_2}}\int_{\mathbb H^n}K(e,y)\left(\int_0^\infty|f(r)|^pr^{Q-1}dr\right)^{\frac{1}{p}}|y|_h^{-Q/p}dy\\
&=\omega_Q^{\frac{1}{\bar{p}_2}-\frac{1}{\bar{p}_1}}\int_{\mathbb H^n}K(e,y)|y|_h^{-Q/p}dy\|f\|_{L_{|x|_h}^pL_\theta^{\bar{p}_2}(\mathbb H^n)}.
\end{align*}
Set$f_j=|x|_h^{-Q/p}$,then we have
\begin{align*}
T(f)(x)&=\int_{\mathbb H^n}K(e,y)|y|_h^{-Q/p}dy|x|_h^{-Q/p}\\
\|T(f)(x)\|_{L_{|x|_h}^pL_\theta^{\bar{p}_2}(\mathbb H^n)}&=\omega_Q^{\frac{1}{\bar{p}_2}-\frac{1}{\bar{p}_1}}\int_{\mathbb H^n}K(e,y)|y|_h^{-Q/p}dy\|f_j\|_{L_{|x|_h}^pL_\theta^{\bar{p}_2}(\mathbb H^n)}
\end{align*}
The Theorem 2 has been completed.
\end{proof}

\section{Optimal Bounds of Hilbert Operators in Mixed Radial Angular Spaces}
\begin{defn}\label{defn1}
Hilbert operator is a basic generalization of classical Hilbert inequality, $m$-linear $n$-dimensional Hilbert operator is defined as
$$
T_{m}(f_1,...,f_m)(x)=\int_{\mathbb R^{nm}}\frac{f_{1}(y_1)\cdots f_{m}(y_m)}{(|x|^{n}+|y_1|^{n}+\cdots+|y_m|^{n})^{m}}dy_{1}\cdots dy_{m},x\in\mathbb R^{n}\backslash\{0\}.
$$
Firstly, we give the definition of $1$-linear $1$-dimensional Hilbert operator.
$$\int_0^\infty \frac {f(t)}{(x+t)}dt.$$
Let $t=xy$,that is $dy=d\frac{t}{x}=\frac{1}{x}dt$, then the above formula can get another form,
  $$
  \int_0^\infty \frac{f(|x|y)}{(1+y)}dy.
  $$
From this, we can generalize the form of $1$-linear $n$-dimensional Hilbert operator,
  $$
  T(f)(x)=\int_{\mathbb R^{n}}\frac{f(|x|y)}{(1+|y|^{n})}dy
  $$
  In the same way, $m$-linear $n$-dimensional Hilbert operator is generalized.

\begin{align*}
T_{m}(f_1,...,f_m)(x)&=\int_{\mathbb R^{nm}}\frac{f_1(y_1)...f_m(y_m)}{(1+|y_1|^n+...+
|y_m|^n)^m}dy_1...dy_m\\
&=\int_{\mathbb R^n}...\int_{\mathbb R^n}\frac{f_1(|x|y_1)...f_m(y_m)}{(1+|y_1|^n+...+|y_m|^n)^m}dy_1...dy_m.
\end{align*}
It's expansion form in Heisenberg Group is as follows
\begin{align*}
\int_{\mathbb H^{nm}}\frac{f_1(y_1)...f_m(y_m)}{|x|_h^Q+|y_1|_h^Q+...+|y_m|_h^Q}.
\end{align*}
\end{defn}
\begin{proof}
The sharp norm of Hilbert operator in mixed radial angular space
style
\begin{align*}
D_m=\omega_Q^{\frac{1}{\bar{p}_2}-\frac{1}{\bar{p}_1}}\omega_Q^m\int_0^\infty...\int_0^\infty\frac{\prod_{i=1}^m|r|^{Q-\frac{Q}{p_i}-1}}{(1+\prod_{i=1}^m|s|)^m}dr_1...dr_m.
\end{align*}
By using the spherical coordinate formula, the calculation can be continued.
\begin{align*}
D_m&=\omega_Q^{\frac{1}{\bar{p}_2}-\frac{1}{\bar{p}_1}}\omega_Q^m\int_0^\infty...\int_0^\infty\frac{\prod_{i=1}^m|r|^{Q-\frac{Q}{p_i}-1}}{(1+\sum_{i=1}^m|s|)^m}ds_1...ds_m\\
&=\frac{\omega_Q^{\frac{1}{\bar{p}_2}-\frac{1}{\bar{p}_1}}\omega_Q^m}{Q^m}\int_0^\infty...\int_0^\infty\frac{\prod_{i=1}^m|s|^{-\frac{1}{p_i}}}{(1+\sum_{i=1}^m|s|)^m}ds_1...ds_m\\
&=\frac{\omega_Q^{\frac{1}{\bar{p}_2}-\frac{1}{\bar{p}_1}}\omega_Q^m}{Q^m}\int_0^\infty...\int_0^\infty\frac{\prod_{i=1}^m|s|^{-\beta_i}}{(1+\sum_{i=1}^m|s|)^m}ds_1...ds_m,
\end{align*}
where,$$|s|=|r|^Q,\beta_i=\frac{1}{p_i}.$$

Using the definitions and properties of Beta function and Gamme function, we can further get,set $$t=\sum_{i=1}^m|s|,$$
then,
\begin{align*}
\int_0^\infty\frac{1}{(1+t)^at^{\beta}}dt=\int_0^1(1-t)^{-\beta}t^{a+\beta-2}dt=B(1-\beta,a+\beta-1).\\
\end{align*}
Let $t_m=(1+t_1+...+t_{m-1})q_m$ ,then we can have
\begin{align*}
&I_m(a,\beta_1,...,\beta_m)\\
&=\int_0^\infty...\int_0^\infty\frac{t_1^{-\beta_1}...t_{m-1}^{-\beta_{m-1}}}{(1+t_1+...+t_{m-1})^{a-1+{\beta_m}}}dt_1...dt_{m-1}\int_0^\infty\frac{1}{(1+q_m)^aq_m^{\beta_m}}dq_m\\
&=B(1-\beta_m,a+\beta_m-1)I_m(a-1+\beta_m,\beta_1,...,\beta_m)\\
&=\frac{\prod_{i=1}^m\Gamma(1-\beta_i)\Gamma(a-m+\sum_{i=1}^m\beta_i)}{\Gamma(a)}
\end{align*}
Giving $\Omega_Q=\frac{2\pi^{n+\frac{1}{2}}\Gamma(\frac{n}{2})}{(n+1)\Gamma(\frac{n+1}{2})}$, we can get
\begin{align*}
&\|p_m^{h_*}\|_{\prod_{i=1}^m\mathbb H_{a_i}^\infty(\mathbb H^n)\rightarrow\mathbb H_{a}^\infty(\mathbb H^n)}\\
&=\frac{\omega_Q^{\frac{1}{\bar{p}_2}-\frac{1}{\bar{p}_1}}\omega_Q^m}{Q^m}\frac{\prod_{i=1}^m\Gamma(1-\frac{a_i}{Q})\Gamma(\frac{a}{Q})}{\Gamma(m)}\\
&=\left(\frac{2\pi^{n+\frac{1}{2}}\Gamma(\frac{n}{2})}{(n+1)\Gamma(\frac{n+1}{2})}\right)^m\left(Q\frac{2\pi^{n+\frac{1}{2}}\Gamma(\frac{n}{2})}{(n+1)\Gamma(\frac{n+1}{2})}\right)^{\frac{1}{\bar{p}_2}-\frac{1}{\bar{p}_1}}\frac{\prod_{i=1}^m\Gamma(1-\frac{a_i}{Q})\Gamma(\frac{a}{Q})}{\Gamma(m)}
\end{align*}
Prove completion.
\end{proof}
\section{Optimal Bound of Hardy-Littlewood-P\'{o}lya Operator in Mixed Radial Angular Space}
\begin{thm}\label{thm3}
For Hardy-Littlewood-P\'{o}lya operator $H$,it is bounded from $L_{|x|_h}^pL_{\theta}^{\bar{p}_1}(\mathbb H^n)$ to $L_{|x|_h}^pL_{\theta}^{\bar{p}_2}(\mathbb H^n)$.
\begin{align*}
\|H(f)\|_{L_{|x|_h}^pL_{\theta}^{\bar{p}_1}(\mathbb H^n)\rightarrow L_{|x|_h}^pL_{\theta}^{\bar{p}_2}(\mathbb H^n)}\leq G.
\end{align*}
Where the sharp constant is
\begin{align*}
G=\omega_Q^{1/\bar{p}_2-1/\bar{p}_1}\int_{\mathbb H^n}\frac{|y|_h^{-Q/p}}{\max(1,|y|_h^Q)}dy.
\end{align*}
In addition, if operator $H$ is bounded, we can get the norm of operator $H$,
\begin{align*}
\|H(f)\|_{L_{|x|_h}^pL_{\theta}^{\bar{p}_1}(\mathbb H^n)\rightarrow L_{|x|_h}^pL_{\theta}^{\bar{p}_2}(\mathbb H^n)}=G.
\end{align*}
\begin{defn}\label{cor1}
If Theorem 3 is established then Hardy-Littlewood-P$\acute{o}$lya operator $H$is bounded, so we have
\begin{align*}
\|H(f)\|_{L_{|x|_h}^pL_{\theta}^{\bar{p}_1}(\mathbb H^n)\rightarrow L_{|x|_h}^pL_{\theta}^{\bar{p}_2}(\mathbb H^n)}=\omega_Q^{1/\bar{p}_2-1/\bar{p}_1}\frac{\omega_QQ}{(Q-Q/p)Q/p}.
\end{align*}
\end{defn}
\end{thm}
\begin{proof}
Using Minkowski's inequality, we can get
\begin{align*}
\|H(f)\|_{L_{|x|_h}^pL_{\theta}^{\bar{p}_2}(\mathbb H^n)}&=\omega_Q^{1/\bar{p}_2}\left(\int_0^\infty\left|\int_{\mathbb H^n}\frac{1}{\max(1,|y|_h^Q)}f(\delta_r,y)dy\right|^pr^{Q-1}dr\right)^{1/p}\\
&\leq\omega_Q^{1/\bar{p}_2}\int_{\mathbb H^n}\left(\int_0^\infty|\frac{1}{\max(1,|y|_h^Q)}f(\delta_ry)|^pr^{Q-1}dr\right)^{1/p}dy\\
&=\omega_Q^{1/\bar{p}_2}\int_{\mathbb H^n}\left(\int_0^\infty\left|\frac{1}{\max(1,|y|_h^Q)}f(\delta_yy)\right|^pr^{Q-1}dr\right)^{1/p}dy\\
&=\omega_Q^{1/\bar{p}_2-1/\bar{p}_1}\int_{\mathbb H^n}\left(\int_0^\infty\omega_Q^{1/\bar{p}_1}\left|\frac{1}{\max(1,|y|_h^Q)}f(r)\right|^pr^{Q-1}dr\right)^{1/p}|y|_h^{-Q/p}dy\\
&=\omega_Q^{1/\bar{p}_2-1/\bar{p}_1}\int_{\mathbb H^n}\frac{|y|_h^{-Q/p}}{\max(1,|y|_h^Q)}dy\left(\int_0^\infty\omega_Q^{1/\bar{p}_1}\left|\frac{1}{\max(1,|y|_h^Q)}f(r)\right|^pr^{Q-1}dr\right)^{1/p}\\
&=\omega_Q^{1/\bar{p}_2-1/\bar{p}_1}\int_{\mathbb H^n}\frac{|y|_h^{-Q/p}}{\max(1,|y|_h^Q)}dy\|f\|_{L_{|x|_h}^pL_{\theta}^{\bar{p}_2}(\mathbb H^n)}.
\end{align*}
Set$f(x)=|x|_h^{-Q/p}$,with a simple exchange, we can get
$$
H(f(x))=G|x|_h^{-Q/p}
$$
and
\begin{align*}
\|H(f)\|_{L_{|x|_h}^pL_{\theta}^{\bar{p}_1}(\mathbb H^n)\rightarrow L_{|x|_h}^pL_{\theta}^{\bar{p}_2}(\mathbb H^n)}=G\|f\|_{L_{|x|_h}^pL_{\theta}^{\bar{p}_2}(\mathbb H^n)}
\end{align*}

Let's find $$\int_{\mathbb H^n}\frac{|y|_h^{-Q/p}}{\max(1,|y|_h^Q)}dy$$
Then we have,
\begin{align*}
\int_{\mathbb H^n}\frac{|y|_h^{-Q/p}}{\max(1,|y|_h^Q)}dy&=\int_{|y|_h<1}|y|_h^{-Q/p}dy+\int_{|y|_h>1}|y|_h^{-Q/p-Q}dy,\\
&=I_0+I_1,\\
I_0&=\int_{|y|_h<1}|y|_h^{-Q/p}dy\\
&=\frac{\omega_Q}{Q-Q/p},\\
I_1&=\int_{|y|_h>1}|y|_h^{-Q/p-Q}dy\\
&=\frac{\omega_Q}{Q/p},\\
I_0+I_1&=\frac{\omega_Q}{Q-Q/p}+\frac{\omega_Q}{Q/p}\\
&=\frac{\omega_Q(Q/p)+\omega_Q(Q-Q/p)}{(Q-Q/p)Q/p}\\
&=\frac{\omega_QQ}{(Q-Q/p)Q/p},
\end{align*}
then
\begin{align*}
\int_{\mathbb H^n}\frac{|y|_h^{-Q/p}}{\max(1,|y|_h^Q)}dy=\frac{\omega_QQ}{(Q-Q/p)Q/p},
\end{align*}
so
\begin{align*}
\omega_Q^{1/\bar{p}_2-1/\bar{p}_1}\int_{\mathbb H^n}\frac{|y|_h^{-Q/p}}{\max(1,|y|_h^Q)}dy=\omega_Q^{1/\bar{p}_2-1/\bar{p}_1}\frac{\omega_QQ}{(Q-Q/p)Q/p},
\end{align*}
Prove completion.
\end{proof}

\begin{flushleft}

	\vspace{0.3cm}\textsc{Xiang Li\\School of Science\\Shandong Jianzhu University\\Jinan, 250000\\P. R. China}

\emph{E-mail address}: \textsf{lixiang162@mails.ucas.ac.cn}

	\vspace{0.3cm}\textsc{Huan Liang\\School of Science\\Shandong Jianzhu University \\Jinan, 250000\\P. R. China}
	
	\emph{E-mail address}: \textsf{lianghuan202211@163.com}
	
	\vspace{0.3cm}\textsc{Shaozhuang Xu\\School of Science\\Shandong Jianzhu University \\Jinan, 250000\\P. R. China}

\emph{E-mail address}: \textsf{2417332391@qq.com}

	\vspace{0.3cm}\textsc{Dunyan Yan\\School of Mathematical Sciences\\University of Chinese Academy of Sciences\\Beijing, 100049\\P. R. China}
	
	\emph{E-mail address}: \textsf{ydunyan@163.com}

\end{flushleft}

\end{document}